\documentclass[12pt]{amsart}
\usepackage{amsfonts}
\usepackage{amssymb}
\usepackage[all,poly,graph]{xy}
\usepackage{graphicx}

\font\got=eufm10
\font\gots=eufm10 at 7pt

\def\g2{ \hbox{\got g}_2}
\def\f4{\hbox{\got f}_4}
\def\d4{\hbox{\got d}_4}
\def\fs4{\hbox{\gots f}_4}
\def\F4{\hbox{\got F}_4}
\def\Fs4{\hbox{\gots F}_4}

\def\es7{ \hbox{\got e}_7}
\def\es8{ \hbox{\got e}_8}

\def\Q{ {\mathbb Q}}

\def\F{\mathbb F}
\def\R{\mathbb R}
\def\S{\mathbb S}
\def\br{\mathbb R}
\def\bs{\mathbb S}

\newcommand{\map}{\operatorname{{\rm map}}}

\def\al{\ifcase\xypolynode\or F \or A\or B\or C\or D\or G\fi}
\def\ala{\ifcase\xypolynode\or a \or b\or c\or d\or g\or f\fi}

\newtheorem{te}{Theorem}
\newtheorem{pr}{Proposition}

\newtheorem{co}{Corollary}

\title[Rational maps from  Euclidean configuration spaces to spheres]{Rational maps  from  Euclidean configuration spaces to spheres}

\usepackage{graphicx}
\begin{document}

\setlength{\unitlength}{0.06in}

  \author{Urtzi Buijs, Antonio Garv{\' i}n and Aniceto Murillo}

\begin{abstract} In this note we give an algorithm to determine the rational homotopy type of the free and pointed mapping spaces $\map(F(\br^m,k),\bs^n)$ and $\map^*(F(\br^m,k),\bs^n)$. An explicit description of these spaces is given for $k=3$. The general case for $n$ odd is also presented as an immediate consequence of the rational version of a classical result of Thom.
\end{abstract}
\maketitle



\section{Introduction}  We are interested in determining the rational homotopy type of the spaces $\map(F(\br^m,k),\bs^n)$ and $\map^*(F(\br^m,k),\bs^n)$ of free and pointed continuous maps from the configuration spaces of $k$ particles in $\R^m$ to the $n$-dimensional sphere. These spaces are useful. For instance, since $F(\R^m,2)\simeq \S^{m-1}$, they include mapping  spaces between spheres whose rational homotopy type have already been described in \cite{BM13}. Also, recall that the generalized Randakumar and Ramana Rao problem \cite{blazieg,nanra,nanrao2},  a strong generalization of the classical Borsuk-Ulam Theorem, asks whether a convex $m$-dimensional polytope can be partitioned into $k$ convex pieces on which $m-1$ continuous functions are equalized ($m,k\ge 2$). Whenever $k$ is a prime power, an affirmative answer \cite[Thm. 1.2]{blazieg}  follows  from the non existence of $\Sigma_k$-equivariant maps
$
F(\R^m,k)\longrightarrow S(W_k^{\oplus m-1})$. Here,  $W_k$ is the hyperplane of $\R^k$ of equation $x_1+\cdots+x_k=0$ and $S(W_k^{\oplus m-1})$ is the unit sphere on the direct sum of $m-1$ copies of $W_k$.  Observe that the symmetric group $\Sigma_k$ naturally acts on both spaces by permuting coordinates and columns respectively, and $S(W_k^{\oplus m-1})$ is just a special $\Sigma_k$ representation of $\S^{(m-1)(k-1)-1}$.

For $k=3$ we obtain the following decomposition in which, for simplicity in the notation, we denote $M(m,n)=\map(F(\mathbb R^m , 3), \S^n)$ and  $M^*(m,n)=\map^*(F(\mathbb R^m , 3), \S^n)$.

\begin{te}\label{theo}
(i) For $n$ odd and any $m\ge 2$,
$$
M(m,n)\simeq_\mathbb Q
\begin{cases}
&\S^n\times K(  \mathbb Q , n-(m-1))^3\times K(  \mathbb Q , n-2(m-1))^2 ,\, \text{if $n>2(m-1)$,}\\
&\S^n\times K(  \mathbb Q , n-(m-1))^3 ,\quad\text{if $m-1<n<2(m-1)$,}\\
&\displaystyle\bigsqcup _{\mathbb N} \S^n,\quad\text{if $n= m-1$,}\\
&\S^n,\quad\text{if $n< m-1$.}\\
\end{cases}
$$

$$
M^*(m,n)\simeq_\mathbb Q
\begin{cases}
&K(  \mathbb Q , n-(m-1))^3\times K(  \mathbb Q , n-2(m-1))^2 ,\, \text{if $n>2(m-1)$,}\\
& K(  \mathbb Q , n-(m-1))^3 ,\quad\text{if $m-1<n<2(m-1)$,}\\
&\displaystyle\bigsqcup _{\mathbb N} *,\quad\text{if $n= m-1$,}\\
&*,\quad\text{if $n< m-1$.}\\
\end{cases}
$$

\medskip

(ii) For $n=2$  and any $m\ge 2$,

\medskip

$$
M(m,2)\simeq_\mathbb Q
\begin{cases}
&\S^{2},\quad\text{if $m>4$,}\\
&\displaystyle\bigsqcup _{\mathbb N} \S^{2},\quad\text{if $m=4$,}\\
&(\S^1)^3\times \S^2 \sqcup\displaystyle\bigsqcup _{\mathbb N} (\S^{1})^2\times \S^3 ,\quad\text{if $m=3$,}\\
&   X\sqcup\displaystyle\bigsqcup _{\mathbb N} \S^1\times H_e\times K(  \mathbb Q , 2)^3\times \S^3   ,\quad\text{if $m=2$.}\\
\end{cases}
$$

$$
M^*(m,2)\simeq_\mathbb Q
\begin{cases}
&*,\quad\text{if $m>4$,}\\
&\displaystyle\bigsqcup _{\mathbb N} *,\quad\text{if $m=4$,}\\
& \displaystyle\bigsqcup _{\mathbb N} (\S^{1})^3 ,\quad\text{if $m=3$,}\\
&  \displaystyle\bigsqcup _{\mathbb N}Y\times K(  \mathbb Q , 2)^3\,  ,\quad\text{if $m=2$.}\\
\end{cases}
$$

\medskip

 (iii) For $n$ even greater than $2$ and  $m\ge 2$,

\medskip

$$
M(m,n)\simeq_\mathbb Q
\begin{cases}
&\S^{n},\quad\text{if $m>2n$,}\\
& \displaystyle\bigsqcup _{\mathbb N}\S^n      ,\quad\text{if $m=2n$,}\\
& K(\mathbb Q, 2n-m)\times \S^n           ,\quad\text{if $n+2\leq m\leq 2n-1$}.\\
\end{cases}
$$

$$
M^*(m,n)\simeq_\mathbb Q
\begin{cases}
&  *,\quad\text{if $m>2n $,}\\
&      \displaystyle\bigsqcup _{\mathbb N} *                                                         ,\quad\text{if $m=2n$,}\\
& K(\mathbb Q, 2n-m)          ,\quad\text{if $n+2\leq m\leq 2n-1$}.\\

\end{cases}
$$

\end{te}
Here $H_e$ is the Heisenberg manifold, $X$ is a rational space which is the total space  in a rational fibration of the form
$$
(\S^1)^2\times K(\mathbb Q, 2)^3\rightarrow X\rightarrow (\S^1)^3\times \S^2 ,
$$
and $Y$ is the nilmanifold \cite{ha} whose minimal model is
$$
(\Lambda (a_1, b_1, c_1, x_1, y_1),d),\quad  dx_1=a_1b_1, \quad dy_1=b_1c_1,
$$
with subscripts indicating degree (see Section 3 for details). As usual,  $\simeq_\mathbb Q$  means ``rationally equivalent to'' and $\bigsqcup$ denotes  disjoint union.

The method used in the proof  may well serve as an algorithm to compute $\map(F(\R^m,k),\S^n)$ and $\map^*(F(\R^m,k),\S^n)$ for given integers $m,k\ge 2$ and $n\ge 1$. However, the general case for more than three particles  for $n$ even, does not produce such a straight decomposition.
Nevertheless,  whenever $n$ is odd, $\bs^n$ is rationally an $H$-space and the spaces $\map(F(\br^m,k),\bs^n)$, $\map^*(F(\br^m,k),\bs^n)$ can be easily decomposed as  products of Eilenberg-MacLane spaces in view of the rational version, both free and pointed, of the classical work of Thom \cite{thom,Hae82}, see Proposition \ref{propo}:

\begin{te}\label{elprime} Denote $M(m,k,n)=\map(F(\br^m,k),\bs^n)$ and  $M^*(m,k,n)=\map^*(F(\br^m,k),\bs^n)$. Then:
$$
M(m,k,n)\simeq_\mathbb Q
\begin{cases}
&\displaystyle\prod _{j=0}^{k-1}K(  \mathbb Q , n-j(m-1))^{k\brack k-j} ,\qquad  \text{if $n>(k-1)(m-1)$},\\
&\displaystyle\bigsqcup _{\mathbb N}\left ( \prod _{j=0}^{l-1} K(  \mathbb Q , n-j(m-1))^{k\brack k-j} \right ) ,\,  \text{if $n=l(m-1)$},\\
&\hspace {7cm}  1\leq l\leq k-1,\\
&\displaystyle\prod _{j=0}^{l}K(  \mathbb Q , n-j(m-1))^{k\brack k-j},\, \text{if $l(m-1)<n<(l+1)(m-1)$},\\
&\hspace {7cm}  1\leq l\leq k-2,\\
&\S^n,\quad\text{if $n< m-1$.}\\
\end{cases}
$$

$$
M^*(m,k,n)\simeq_\mathbb Q
\begin{cases}
&\displaystyle\prod _{j=1}^{k-1}K(  \mathbb Q , n-j(m-1))^{k\brack k-j} ,\qquad  \text{if $n>(k-1)(m-1)$},\\
&\displaystyle\bigsqcup _{\mathbb N}\left ( \prod _{j=1}^{l-1} K(  \mathbb Q , n-j(m-1))^{k\brack k-j} \right ) ,\,  \text{if $n=l(m-1)$},\\
&\hspace {7cm}  1\leq l\leq k-1,\\
&\displaystyle\prod _{j=1}^{l}K(  \mathbb Q , n-j(m-1))^{k\brack k-j},\, \text{if $l(m-1)<n<(l+1)(m-1)$},\\
&\hspace {7cm}  1\leq l\leq k-2,\\
&*,\quad\text{if $n< m-1$.}\\
\end{cases}
$$

\end{te}
Here, as in \cite{Knuth}, the brackets $k\brack k-j$ represent the unsigned Stirling numbers of the first kind.

 As an illustrative example, for the case including the generalized Randakumar and Ramana Rao problem, we get directly:

 \begin{co}\label{coro} If either $m$ or $k$ is an odd number, then:

For $m\geq 3$,
$$
\map(F(\mathbb R^m , k), \S^{(m-1)(k-1)-1})\simeq_\mathbb Q
 \prod _{j=0}^{k-2} K(  \mathbb Q , (k-(j+1))(m-1)-1 )^{k\brack k-j} ,
$$
$$
\map^*(F(\mathbb R^2 , k), \S^{(m-1)(k-1)-1})\simeq_\mathbb Q
 \prod _{j=1}^{k-2} K(  \mathbb Q , (k-(j+1))(m-1)-1 )^{k\brack k-j} .
$$

For $m=2$,
$$
\map(F(\mathbb R^2 , k), \S^{k-2})\simeq_\mathbb Q
\displaystyle\bigsqcup _{\mathbb N}\left ( \prod _{j=0}^{k-3} K(  \mathbb Q , k-(2+j))^{k\brack k-j} \right ) ,
$$
$$
\map^*(F(\mathbb R^2 , k), \S^{k-2})\simeq_\mathbb Q
\displaystyle\bigsqcup _{\mathbb N}\left ( \prod _{j=1}^{k-3} K(  \mathbb Q , k-(2+j))^{k\brack k-j} \right ).\eqno{\square}
$$
\end{co}

\section{Preliminaries}
\label{prelim}

We will use basic results from rational homotopy theory
for which~\cite{FHT00} has become a standard reference.
Via the classical  adjoint functors between the  categories
of commutative differential graded algebras (CDGA's henceforth) over $\Q$ which is always assumed to be the ground field, and  simplicial sets, given by
 piecewise linear forms and realization,
$$\mbox{SSets}
\renewcommand{\arraystretch}{0.5}\begin{array}[t]{l}
\stackrel{A_{PL}}{\to}\\
\leftarrow\\\langle\, \rangle\end{array}
\renewcommand{\arraystretch}{1}\, \mbox{CDGA},$$
one has the notion of (Sullivan) model of a non necessarily connected space $Z$ such that all its
components are nilpotent \cite{BM06}:
By such a model we mean  a cofibrant
$\mathbb{Z}$-graded free commutative differential graded algebra
whose simplicial realization has the same
homotopy type as the Milnor simplicial approximation of the rationalization of $Z$.

If $(\Lambda W,d)$ is a  model of  $Z$ in this sense and  $u\colon
\Lambda W\to \mathbb Q $ the model of a $0$-simplex of $Z$, consider the differential ideal  $K_u$ generated by  $A_1 \cup A_2 \cup A_3$, being
$$A_1=W^{< 0},\ A_2=d W^0,\ A_3=\{ \alpha -u(\alpha ): \alpha \in W^0 \}. $$
Then $(\Lambda W, d )/K_u$ is again a free commutative differential graded algebra    of the form $(\Lambda (\overline{W}^1\oplus W^{\geq 2}),d_u)$ in which $\overline{W}^1$ is a complement in $W^1$
of $d(W^0)$ up to identifications given by $A_1$ and $A_3$, see \cite[S4]{BM06} for details. Then \cite[4.3]{BS97}, $(\Lambda (\overline{W}^1\oplus W^{\geq 2}),d_u)$ is a Sullivan model of the path component of $Z$ containing
the fixed $0$-simplex.

In particular, if $X$ is a nilpotent finite
CW-complex and $Y$ is a finite type
CW-complex then   the components of the free and pointed mapping spaces $\map(X,Y)$ and $\map^*(X,Y)$ are nilpotent \cite{HMR75} and the above applies. We briefly recall  the Haefliger model \cite{Hae82} of these  spaces and its components following the presentation in
 \cite{BS97,BM06}.

 Let $B$ be a finite dimensional  commutative differential graded algebra model of $X$ and let
$A=(\Lambda V, d)$ be a Sullivan  model of $Y$. We
denote by $B^\sharp$ the differential graded coalgebra dual of $B$, $B^\sharp=Hom (B, \mathbb Q)$,  with the grading
$({B^\sharp})^{-n}=(B^\sharp)_{n}=Hom(B^n,\mathbb{Q})$. Consider the free
commutative differential graded algebra $\Lambda ( A\otimes {B^\sharp})$ generated by
the $\mathbb Z$-graded vector space  $A\otimes {B^\sharp}$, with
the differential $d$ induced by the one  on $A$ and $B^\sharp$.
 Let $I\subset \Lambda (A\otimes {B^\sharp})$ be the differential ideal generated by
$1\otimes 1-1$, and the elements of the form
$$
a_1a_2\otimes \beta -\sum_j(-1)^{|a_2||\beta_j'|}(a_1\otimes
\beta_j')(a_2\otimes \beta _j''),\quad a_1,a_2\in A,\, \beta\in B^\sharp,
$$
where the coproduct on $\beta$ is,
 $\Delta \beta =\sum_j \beta_j'\otimes
\beta_j''$. The inclusion $V\otimes {B^\sharp}\hookrightarrow A\otimes {B^\sharp}$ induces an isomorphism of graded algebras
$$
\rho \colon \Lambda ( V\otimes {B^\sharp})\stackrel{\cong}{\longrightarrow} \Lambda (A\otimes {B^\sharp})/I
$$
and thus $\widetilde{d}=\rho^{-1}d\rho $ defines a differential in $\Lambda
(V\otimes {B^\sharp})$.
We can do the same construction taking $(B_+)^\sharp$ (elements of $B^\sharp$ of negative degree) instead of $B^\sharp$,  and taking $\bar \Delta$ (the reduced coproduct) instead of $\Delta$.

 Then \cite{BS97, BM06}, the commutative differential graded algebra $(\Lambda
(V\otimes {B^\sharp}), \widetilde{d} )$ is a model of  $\map(X,Y)$,
and
the commutative differential graded algebra  $(\Lambda (V\otimes B_+^\sharp
), \widetilde{d} )$ is a model of $\map^* (X,Y)$.  \label{BSH}

 Now, let  $\varphi\colon (\Lambda V,d)\to (B,\delta )$ a model of a given map $f\colon X\to Y$. The morphism $\varphi$ induces a natural augmentation  denoted also by
$\varphi\colon \bigl(\Lambda(V\otimes B^\sharp),\widetilde d\bigr)\to\mathbb Q$ which can be thought as the model of the $0$-simplex of the mapping space representing $f$. Applying the process above
we obtain the  Sullivan algebra
$$
\bigl(\Lambda \bigl(\overline{V\otimes B^\sharp}^1\otimes (V\otimes B^\sharp)^{\ge 2}\bigr),\widetilde d_\varphi\bigr)$$
which is
a Sullivan model of the component $\map_f(X,Y)$ of the free mapping space containing $f$ \cite{BM06}.
In the same way,
$$
\bigl(\Lambda \bigl(\overline{V\otimes B_+^\sharp}^1\otimes (V\otimes B_+^\sharp)^{\ge 2}\bigr),\widetilde d_\varphi\bigr)$$
is a Sullivan model of $\map^*_f(X,Y)$.

The next result will be used in next sections. It may be considered a rational reformulation of the classical decomposition of Thom \cite{thom}, see also \cite{Hae82}.

\begin{pr}\label{propo}
Let $X$ be a formal finite nilpotent complex and let $Y$ be of the rational homotopy type of a finite type H-space. For $j\geq 0$ let
$$N_j=\sum_{r-s=j} \dim\Pi _r(Y)\otimes\mathbb Q \cdot \dim H^s(X; \mathbb Q), $$
$$N'_j=\sum_{r-s=j,\,  s\neq 0} \dim\Pi _r(Y)\otimes\mathbb Q \cdot \dim H^s(X; \mathbb Q). $$
Then,
 $$\map(X,Y)\simeq_\mathbb Q \begin{cases} \prod_{j\geq 1}K(\mathbb Q , j)^{N_j}& \text{if $N_0=0$},\\
 \bigsqcup_\mathbb N\left( \prod_{j\geq 1}K(\mathbb Q , j)^{N_j} \right) & \text{if $N_0\not=0$}.
 \end{cases}
$$
$$ \map^*(X,Y)\simeq_\mathbb Q \begin{cases} \prod_{j\geq 1}K(\mathbb Q , j)^{N'_j} & \text{if $N'_0=0$},\\
\bigsqcup_\mathbb N\left( \prod_{j\geq 1}K(\mathbb Q , j)^{N'_j} \right) & \text{if $N'_0\not=0$}.
\end{cases}
$$
\end{pr}

\begin{proof}
As $X$ is a formal space, $B=(H^*(X;\Q) ,0)$ is a model of $X$. On the other hand the minimal model of the  H-space $Y$ is of the form
$A=(\Lambda V, 0)$. Then, $(\Lambda( V\otimes B^{\sharp}), 0)$ is  a model of $\map(X,Y)$.

Observe that for any $j$,
$$
\begin{aligned}\dim(V\otimes B^{\sharp})^j  &= \sum_{r+s=j} =\dim V^r \cdot \dim(B^{\sharp})^s\\&= \sum_{r-s=j} \dim\Pi _r(Y)\otimes\mathbb Q \cdot \dim H^s(X; \mathbb Q)=N_j ,\end{aligned}$$
$$
\begin{aligned}\dim(V\otimes B_+^{\sharp})^j  &= \sum_{r+s=j,\, s\neq 0} \dim V^r \cdot \dim(B_+^{\sharp})^s\\&=
\sum_{r-s=j,\, s\neq 0} \dim\Pi _r(Y)\otimes\mathbb Q \cdot \dim  H^s(X; \mathbb Q)=N'_j .\end{aligned}$$
Now, both in the free or pointed case,  there is only one component as long as $(V\otimes B^{\sharp})^0=0$ or $(V\otimes B_+^{\sharp})^0=0$, that is, whenever $N_0=0$ or $N'_0=0$. Otherwise, as the differential is trivial, there are   a countable number of components,  as non homotopic  augmentations in $(V\otimes B^{\sharp})^0$ or $(V\otimes B_+^{\sharp})^0$.  On the other hand, again by the triviality of the differential, it is clear that each component  is of the homotopy type of  $\prod_{j\geq 1}K(\mathbb Q , j)^{N_j}$ in the free case and  $\prod_{j\geq 1}K(\mathbb Q , j)^{N'_j}$ in the pointed case.

\end{proof}

 \section{The proofs}

We first prove Theorem \ref{theo} by applying  the procedure in Section 2 to obtain a model of    $\map(F(\mathbb R^m,3), \S^{n})$. Then, we identify from this model the rational homotopy type of its components.

A CDGA model of the configuration space $F(\mathbb R^m, k)$ is given by its rational cohomology algebra as these spaces are formal  \cite{Kon99}.
It is well known \cite{CLM76} that  $H^*(F(\mathbb R^m, k);\Q)$, is given by
\begin{equation}\label{formula}
H^*(F(\mathbb R^m, k);\Q)=\Lambda (a_{ij})/I,\qquad i\not=j,\quad i,j=1,\dots,k.
\end{equation}
where $\mid a_{ij}\mid=m-1$,  and $I$ is the ideal generated as follows:
$$I=\langle a_{ij}-(-1)^m a_{ji},\quad a_{ij}^2,\quad a_{ij}a_{jr}+a_{jr}a_{ri}+a_{ri}a_{ij}\rangle .$$
For $k=3$, we have
$$H^*(F(\mathbb R^m, 3))=\frac{\Lambda (a_{12}, a_{13}, a_{21}, a_{23},a_{31}, a_{32}) }{I},$$
with  $\mid \overline a_{ij}\mid=m-1$, $\overline{a_{ji}}=(-1)^m\overline{a_{ij}}$, $\overline{a_{ij}}^2=0$  and
$$\overline{ a_{12}}\, \overline{a_{23}}+\overline{a_{23}}\, \overline{a_{31}}+\overline{a_{31}}\, \overline{a_{12}}=0.$$

Then, as a  graded vector space $B=H^*(F(\mathbb R^m, 3))$ is concentrated in degrees $0$, $m-1$ and $2(m-1)$,
$$B=\mathbb Q \, \oplus
\langle \overline{a_{12}} ,  \overline{a_{13}},  \overline{a_{23}}\rangle \, \oplus
 \langle \overline{a_{12}}\, \overline{a_{23}},  \overline{a_{13}}\,  \overline{a_{23}}\rangle  .$$
 Hence, its dual vector space is
 $$B^{\sharp}=H_*(F(\mathbb R^m,3);\mathbb Q)=H_0\oplus H_{m-1}\oplus H_{2(m-1)}$$
where $H_0=\mathbb Q=\langle1\rangle$, $H_{m-1}=\langle \alpha_{12}  ,\alpha_{13}  ,\alpha_{23} \rangle$ and $ H_{2(m-1)}=\langle \alpha_{12, 23} ,\alpha_{13, 23} \rangle$. Here $1, \alpha_{12}  ,\alpha_{13}  ,\alpha_{23}, \alpha_{12, 23} ,\alpha_{13, 23}$ simply denotes the dual basis of
 $1 ,  \overline{a_{12}} ,  \overline{a_{13}}$,  $\overline{a_{23}}, \overline{a_{12}}\, \overline{a_{23}},  \overline{a_{13}}\, \overline{a_{23}} $.

Now, if $n$ is odd, we may apply Proposition \ref{propo} and a straightforward computation proves the assertion  (i) of Theorem \ref{theo}.

From now on we assume $n$ is an even integer and fix the minimal model of $\S^n$ given by
$(\Lambda (x, y), d)$ with $\mid x\mid =n$, $\mid y\mid =2n-1$, $dx=0$ and $dy=x^2$.

We will also need the ring structure of
 $B=H^*(F(\mathbb R^m, 3))$ which is given by the following table:

\medskip
\noindent
\scalebox{0.92}[1]{\begin{tabular}{|c|c|c|c|c|c|c|}
\hline
   & 1& $\overline{a_{12}}$ & $\overline{a_{13}}$ &$\overline{a_{23}}$&$\overline{a_{12}}\,  \overline{a_{23}}$&$\overline{a_{13}}\, \overline{a_{23}}$\\
   \hline
1&  1& $\overline{a_{12}}$ & $\overline{a_{13}}$ &$\overline{a_{23}}$&$\overline{a_{12}}\,  \overline{a_{23}}$&$\overline{a_{13}}\, \overline{a_{23}}$ \\
\hline
$\overline{a_{12}}$ &$\overline{a_{12}}$& 0 &$\overline{a_{12}}\, \overline{a_{23}}-\overline{a_{13}} \, \overline{a_{23}} $  &$\overline{a_{12}} \, \overline{a_{23}} $ &0  &0  \\
\hline
$\overline{a_{13}}$  &$\overline{a_{13}}$& $(-1)^{m-1}\overline{a_{12}}\, \overline{a_{23}}+(-1)^m\overline{a_{13}} \, \overline{a_{23}} $   & 0&  $\overline{a_{13}} \, \overline{a_{23}} $&  0  &0  \\
\hline
$\overline{a_{23}}$  &$\overline{a_{23}}$& $(-1)^{m-1}\overline{a_{12}} \,\overline{a_{23}}$     & $(-1)^{m-1}\overline{a_{13}}\, \overline{a_{23}}$    & 0 &  0  &0\\
\hline
$\overline{a_{12}}\, \overline{a_{23}}$ &$\overline{a_{12}} \,\overline{a_{23}}$ & 0& 0 &  0  &0&0 \\
\hline
$\overline{a_{13}}\, \overline{a_{23}}$  &$\overline{a_{13}}\, \overline{a_{23}}$ & 0& 0 &  0  &0&0\\
\hline
\end{tabular}}
\medskip

From it  one explicitly determines the coproduct $\Delta$ on $B^{\sharp}$:

$$\Delta (1)=1\otimes 1,$$
$$\Delta (\alpha_{12})=1\otimes \alpha_{12}+\alpha_{12}\otimes 1,$$
$$\Delta (\alpha_{13})=1\otimes \alpha_{13}+\alpha_{13}\otimes 1,$$
$$\Delta (\alpha_{23})=1\otimes \alpha_{23}+\alpha_{23}\otimes 1,$$
$$\Delta (\alpha_{12,23})=
1\otimes \alpha_{12,23}+\alpha_{12,23}\otimes 1
+ (-1)^{m+1}\alpha_{12}\otimes \alpha_{23} +\alpha_{23}\otimes \alpha_{12}
+ (-1)^{m+1}\alpha_{12}\otimes \alpha_{13}+\alpha_{13}\otimes \alpha_{12},$$
$$\Delta (\alpha_{13,23})=
1\otimes \alpha_{13,23}+\alpha_{13,23}\otimes 1
+ (-1)^{m+1}\alpha_{13}\otimes \alpha_{23}+\alpha_{23}\otimes \alpha_{13}
+ (-1)^{m}\alpha_{12}\otimes \alpha_{13}-\alpha_{13}\otimes \alpha_{12}.$$

Hence, following the procedure in Section 1, one obtain a model of
 $\map(F(\mathbb R^m, 3) ,\S^n)$ of the form
 $$(\Lambda (V\otimes {B^\sharp}), \widetilde{d} )$$
where
$$\begin{aligned} V\otimes B^{\sharp  }=  &\langle x\otimes 1, x\otimes\alpha_{12}  , x\otimes\alpha_{13}  , x\otimes\alpha_{23},  x\otimes\alpha_{12, 23} , x\otimes\alpha_{13, 23} , \\
  & y\otimes 1, y\otimes\alpha_{12}  , y\otimes\alpha_{13}  , y\otimes\alpha_{23},  y\otimes\alpha_{12, 23} , y\otimes\alpha_{13, 23}\rangle ,  \end{aligned}$$
 in which
 $
 x\otimes 1, x\otimes\alpha_{12}  , x\otimes\alpha_{13}  , x\otimes\alpha_{23},  x\otimes\alpha_{12, 23} , x\otimes\alpha_{13, 23}
 $
 are cycles and
$$ \begin{aligned}
\tilde d(y\otimes 1)=&(x\otimes 1)^2 ,\\
\tilde d(y\otimes \alpha_{12})=& 2(x\otimes 1)(x\otimes \alpha_{12}) ,\\
\tilde d(y\otimes \alpha_{13})=& 2(x\otimes 1)(x\otimes \alpha_{13}) ,\\
\tilde d(y\otimes \alpha_{23})= &2(x\otimes 1)(x\otimes \alpha_{23}) ,\\
\tilde d(y\otimes \alpha_{12, 23})=&2\bigl(  (x\otimes 1)( x\otimes \alpha_{12, 23}) ,\\
&+(-1)^{m+1}( x\otimes \alpha_{12})(x\otimes \alpha_{23})+(-1)^{m+1}(x\otimes  \alpha_{13})( x \otimes \alpha_{23})\bigr) ,\\
\tilde d(y\otimes \alpha_{13, 23})=&2\bigl( (x\otimes 1)( x\otimes \alpha_{13, 23}) ,\\
&+(-1)^{m+1}( x\otimes \alpha_{13})(x\otimes \alpha_{23})+(-1)^{m}(x\otimes  \alpha_{12})( x \otimes \alpha_{13})\bigr) .
   \end{aligned}$$
   To simplify notation write $V\otimes B^{\sharp  }=W$, $\tilde d=d$,
   $$x=x\otimes 1,\quad y=y\otimes 1,$$
   $$p_{i+j-2}=x\otimes\alpha_{ij},\quad q_{i+j-2}=y\otimes\alpha_{ij},$$
   $$ r_{i+j-2}=x\otimes\alpha_{ij,rs},\quad s_{i+j-2}=y\otimes\alpha_{ij, rs}.$$
  Then,
  $$(\Lambda W, {d} )=(\Lambda (x,y,p_1,p_2,p_3,q_1,q_2,q_3,r_1,r_2,s_1,s_2),  d)$$
  where $x$, $p_i$, $r_j$ are cycles ($i=1,2,3$ and $j=1,2$) and
$$
\begin{array}{l}
 d(y)=x^2 ,\\
 d(q_i)= 2xp_i  ,\quad i=1,2,3 \,  ,\\
 d(s_1)=2 ( x r_1   + (-1)^{m+1}p_1p_3 + (-1)^{m+1}p_2 p_3  ),\\
 d(s_2)=2 (  x r_2 +(-1)^{m+1} p_1p_2 + (-1)^{m} p_2p_3 ) .
\end{array}
$$
The degrees of the generators are:
$$
\begin{array}{l}
 \mid x\mid =n ,\\
 \mid y\mid =2n-1 ,\\
 \mid p_i\mid = n-m+1 , \quad i=1,2,3\, , \\
 \mid q_i\mid = 2n-m , \quad i=1,2,3\, , \\
 \mid r_i\mid = n-2m+2 , \quad i=1,2\, , \\
 \mid s_i\mid = 2n-2m+1 , \quad i=1,2\, . \\
\end{array}
$$For the pointed maps, the procedure given in Section 2 produces the following model of $\map^*(F(\mathbb R^m, 3), \S^n)$. Writing $W_+=V\otimes B_+^\sharp$ and with the same notation for the generators, this model is
$$(\Lambda W_+,d)=(\Lambda (p_1,p_2,p_3,q_1,q_2,q_3,r_1,r_2,s_1,s_2), d)$$
where
$$
\begin{array}{l}
 d(s_1)=(-1)^{m+1} 2 ( p_1p_3 + p_2 p_3  ) ,\\
 d(s_2)=(-1)^{m+1} 2 (  p_1p_2 - p_2p_3 ) ,
\end{array}
$$
and the rest of generators are cycles.

We first deal with the case $n=2$ and analyze each component in the cases $m>4$, $m=4$, $m=3$ and $m=2$.

 For $m>4$ we have,
\begin{center}
\begin{tabular}{|c|c|}
\hline
degree& $W$ \\
\hline
$3$& $y$ \\
\hline
$2$& $x$ \\
\hline
$1$&    \\
\hline
$0$&  \\
\hline
$\cdots$&  \\
\hline
$4-m$& $q_1, q_2, q_3$ \\
\hline
$3-m$& $p_1,p_2,p_3$ \\
\hline
$\cdots$&  \\
\hline
$5-2m$& $s_1, s_2$ \\
\hline
$4-2m$& $r_1, r_2$ \\
\hline

\end{tabular}
\end{center}

In this case there are no generators in degree $0$ so the only possible augmentation $\Lambda W \to \mathbb Q $ is the trivial one, that is, there is only one component. Also, there
are no generators in degree $1$, so  projecting over the
the generators of negative degree we obtain the Sullivan model of the $\map(F(\mathbb R ^m , 3), \S^2)$ which turns out to be the minimal model of $\S^2$. For the pointed mapping space observe that $W_+$ is concentrated in negative degrees and therefore $\map^*(F(\mathbb R ^m , 3), \S^2)\simeq _{\mathbb Q} *$.

For  $m=4$ we have,

\begin{center}
\begin{tabular}{|c|c|}
\hline
degree& $W$ \\
\hline
$3$& $y$ \\
\hline
$2$& $x$ \\
\hline
$1$&                        \\
\hline
$0$& $q_1, q_2, q_3$ \\
\hline
$-1$& $p_1,p_2,p_3$ \\
\hline
$-2$&                         \\
\hline
$-3$& $s_1, s_2$ \\
\hline
$-4$& $r_1, r_2$ \\
\hline
\end{tabular}
\end{center}
The existence of generators $q_1,q_2,q_3$ in degree zero provides an augmentation $\varphi \colon (\Lambda W,d)\to\Q$ for each triad $\lambda_1,\lambda_2,\lambda_3$ of rational numbers given by $$\varphi(q_1)=\lambda _1,\quad \varphi(q_2)=\lambda _2,\quad \varphi(q_3)=\lambda _3.$$
Note that different triads produces non homotopic augmentations as they induce different cohomology morphisms.
Therefore, there are a countable number of components in the rationalization of  $\map(F(\mathbb R ^4 , 3), \S^2)$. It is straightforward to check that any of them has  the same  model as $\S^2$, i.e., $\map(F(\mathbb R ^4 , 3), \S^2)\simeq _{\mathbb Q}\bigsqcup _{\mathbb N} \S^2$. As before, $W_+$ is concentrated in non positive degree so that each component of the pointed mapping space is rationally contractible:
$\map^*(F(\mathbb R ^4 , 3), \S^2)\simeq _{\mathbb Q}\bigsqcup _{\mathbb N} *$.

For $m=3$ we have,
\begin{center}
\begin{tabular}{|c|c|}
\hline
degree& $W$ \\
\hline
$3$& $y$ \\
\hline
$2$& $x$ \\
\hline
$1$&    $q_1, q_2, q_3$                   \\
\hline
$0$& $p_1, p_2, p_3$ \\
\hline
$-1$& $s_1, s_2$ \\
\hline
$-2$& $r_1, r_2$ \\
\hline
\end{tabular}
\end{center}
Observe that in this case, an augmentation $\varphi\colon (\Lambda W,d)\to \Q$ is determined also by a triad of rational numbers $\lambda _i=\varphi(p_i$), $i=1,2,3$, which satisfy the equations $\varphi(d s_j)=0$, $j=1,2$. In other words, each augmentation corresponds to a solution of the system,
      $$\left\lbrace
\begin{tabular}{l}
$\lambda _2(\lambda _1-\lambda _3)=0$, \\
  $\lambda _3(\lambda _1+\lambda _2)=0$.
\end{tabular}\right .$$
These are $\{(\lambda,0,0),(0,\lambda,0),(0,0,\lambda),(\lambda,-\lambda,\lambda)\}_{\lambda\in\Q}$
Note also that different solutions correspond to non homotopic augmentations and hence, the mapping space has a countable number of components.

The model of the component corresponding to $\lambda=0$ is
$(\Lambda (q_1, q_2, q_3),0)\otimes (\lambda (x, y), d )$, that is, the model of
 $\S^1\times \S^1\times \S^1\times \S^2 $.
For the rest of the cases straightforward computations provide models of  $ \S^1\times \S^1\times \S^3 $.

Thus,
$\map(F(\mathbb R ^3 , 3), \S^2)\simeq_\mathbb Q (\S^1)^3\times \S^2 \sqcup \bigsqcup _{\mathbb N} (\S^1)^2\times \S^3$.

In the  pointed case we obtain that the model of each component is
$(\Lambda (q_1, q_2, q_3),  0 )$ and therefore
$\map^*(F(\mathbb R ^3 , 3), \S^2)\simeq_\mathbb Q  \bigsqcup _{\mathbb N} (\S^1)^3$.

For $m=2$ we have,
\begin{center}
\begin{tabular}{|c|c|}
\hline
degree& $W$ \\
\hline
$3$& $y$ \\
\hline
$2$& $x,q_1, q_2, q_3$ \\
\hline
$1$&    $p_1, p_2, p_3, s_1, s_2$                   \\
\hline
$0$& $r_1,r_2$ \\
\hline
\end{tabular}
\end{center}
and each augmentation $\varphi\colon(\Lambda W,d)\to\Q$ is determined by a pair of rational numbers $\varphi(r_1)=\lambda_1$ and $\varphi(r_2)=\lambda_2$. According to the procedure in Section 1, the model of the corresponding component is,
$$
(\Lambda W^{\ge 1} ,d)=(\Lambda (x,y,p_1,p_2,p_3,q_1,q_2,q_3,s_1,s_2),d) ,
$$
in which $x,p_1,p_2,p_3$ are cycles and
$$\begin{aligned}
    dy&=x^2 ,   \\
   dq_i&=2xp_i,\,\,i=1,2,3,\\
         ds_1&=2(\lambda _1x-p_1p_3-p_2p_3),\\
 ds_2&=2(\lambda _2x-p_1p_2+p_2p_3).\\
\end{aligned}$$
 For $\lambda_1=\lambda_2=0$ this is the model of
 a rational space $X$
that is the total space of a rational fibration of the form
$$ (\S^1)^2\times K(\mathbb Q, 2)^3\rightarrow X\rightarrow (\S^1)^3\times \S^2 .$$

In the rest of the cases, that is $\lambda_i\not=0$ for some $i=1,2$,  changing basis and discarding the contractible part,
we obtain  the model
$$ (\Lambda  (x_1,  y_1,  z_1,  t_1,   u_2,  v_2,  w_2 ,   u_3 ),d),$$
where subscripts indicates degree,  and all generators are cycles except \break $dt_1=x_1y_1$. The realization is
$$\S^1\times H_e\times K(\mathbb Q, 2)^3\times \S^3,$$
where $H_e$
 is the Heisenberg manifold whose rational model is precisely
 $$(\Lambda  (x_1,  y_1,  z_1,  t_1),d),\quad   dt_1=x_1y_1.$$

Adding up,
$$\map(F(\mathbb R ^2 , 3), \S^2)\simeq_\mathbb Q
X\sqcup \bigsqcup _{\mathbb N} \S^1\times H\times K(\mathbb Q, 2)^3\times \S^3  .$$

 In the based case all the components have the same model,
 $$(\Lambda ( p_1,p_2 ,p_3,s_1,  s_2,q_1, q_2, q_3), d) ,$$
in which the $p_i$'s are cycles and
$$ds_1=\pm 2(p_1p_3+p_2p_3)\, ,\quad ds_2=\pm 2(p_1p_2-p_2p_3) .$$
 It is not difficult to modify this model to get
 $$(\Lambda  (a_1,  b_1,  c_1,   x_1,   y_1),d) \otimes (\Lambda (u_2, v_2, w_2), 0)$$
 where $dx_1=a_1b_1$ and $ dy_1=b_1c_1$,
  here subscripts indicates degree.
  Let $Y$ be the realization of the first factor which has the homotopy type of a nilmanifold \cite{ha}. It is clear that the second one realizes as $K(\mathbb Q, 2)^3$ and thus,
  $$\map^*(F(\mathbb R ^2 , 3), \S^2)\simeq_\mathbb Q
\bigsqcup _{\mathbb N} Y\times K(\mathbb Q, 2)^3  .$$

We now attack the case $n$ even greater or equal than 4. For it, as before, we fix $(\Lambda (x,y),d)$ the minimal model of $\S^n$. We also distinguish different cases:

For $m>2n$,    and arguing as before, we obtain only one component with the same model as $\S^n$ in the free case and $\Q$ in the pointed case. Hence,
$$\map(F(\mathbb R ^m , 3), \S^n)\simeq _{\mathbb Q} \S^n \,\text{and}\,\map^*(F(\mathbb R ^m , 3), \S^n)\simeq _{\mathbb Q} * .$$

For  $m=2n$ the model is of the form
$(\Lambda (x,y),d) \otimes (\Lambda z , 0)$ in the free case 
and $(\Lambda z, 0)$ in the based case, with $z$ of degree $0$. There is trivially a countable number of augmentations and the model of the corresponding component is the model of the $\S^n$ in the free case and contractible in the pointed case. Thus,
$$\map(F(\mathbb R ^m , 3), \S^n)\simeq _{\mathbb Q} \bigsqcup _{\mathbb N} \S^n \,\text{and}\,\map^*(F(\mathbb R ^m , 3), \S^n)\simeq _{\mathbb Q} \bigsqcup _{\mathbb N} * .$$

For $n+2\leq m\leq 2n-1$ the model is of the form
$(\Lambda (x,y),d)\otimes (\Lambda z_{2n-m},0)$
in the free case and $(\Lambda z_{2n-m},0)$ in the pointed one. From here we deduce that
$$
\begin{aligned}\map(F(\mathbb R ^m , 3), \S^n)&\simeq _{\mathbb Q} K(\mathbb Q, 2n-m)\times \S^n,\\\map^*(F(\mathbb R ^m , 3), \S^n)&\simeq _{\mathbb Q}  K(\mathbb Q, 2n-m) .\end{aligned}$$
This completes the proof of Theorem \ref{theo}.

We finish with the proof of  Theorem \ref{elprime}. Recall the explicit description of
$
H^*(F(\mathbb R^m, k);\Q)$ given in (\ref{formula}).
Hence, this cohomology is concentrated in degrees $0$, $m-1$, $2(m-1)$,  \dots , $(k-1)(m-1)$. Now, from the work of Fadell and Neurwith \cite{FN62},  the  Poincar\' e series of
 $F(\mathbb R^m, k)$ is
 $$(1+t^{m-1})(1+2t^{m-1})\cdots (1+(k-1)t^{m-1}) .$$
 From this and the fact \cite{Knuth} that
 $$(1+t)(1+2t)\cdots (1+(k-1)t)=\sum_{j=1}^{k}  {k\brack j } t^j $$
 we obtain the dimensions of the non trivial cohomology of $F(\mathbb R^m , k)$:
 $$\text{dim}\, H^{j(m-1)}(F(\mathbb R^m , k); \Q)={k\brack k-j }\quad \text{for $j=0,\, 1,\, 2,\,\cdots ,\, k-1$}.$$

The proof finishes with a direct computation using Proposition \ref{propo}.

\end{document}